\documentclass[11pt]{amsart}
\usepackage{amssymb, amsfonts, latexsym, amsthm, amsmath, verbatim, enumerate}
\pdfoutput=1
\usepackage{soul} 
\usepackage{paralist}
\usepackage[all]{xy}
\usepackage{amscd}
\usepackage[english]{babel}
\usepackage[margin=1in]{geometry}
\usepackage{tikz-cd}
\usepackage{hyperref}
\usetikzlibrary{arrows}
\theoremstyle{plain}
\newtheorem{thm}{Theorem}[section]
\newtheorem*{thm*}{Theorem}
\newtheorem{prop}[thm]{Proposition}
\newtheorem{cor}[thm]{Corollary}
\newtheorem{lem}[thm]{Lemma}

\newtheorem{examp}[thm]{Example}
\newtheorem{fact}[thm]{Fact}
\theoremstyle{definition}
\newtheorem{defn}[thm]{Definition}

\newcommand{\E}{\mathfrak{E}}
\newcommand{\f}{\mathbb{F}}

\newcommand{\tq}{\mathit{\tilde{H}^*}}
\newcommand{\tu}{\mathit{\tilde{H}}} 

\newcommand{\s}{\mathfrak{s}}
\newcommand{\dzt}{\delta_{\mathbb{Z}/2}}
\newcommand{\dztzf}{\delta_{\mathbb{Z}/4}}

\newcommand{\dg}{\mathit{\delta}_G}
\newcommand{\dglow}{\mathit{\underline{\delta}}_G}
\newcommand{\dgup}{\mathit{\overline{\delta}}_G}
\newcommand{\ttz}{\theta^3_H} 
\newcommand{\pin}{\mathrm{Pin}(2)} 
\newcommand{\hfi}{\mathit{HFI}}
\newcommand{\swf}{\mathit{SWF}}
\newcommand{\swfhg}{\mathit{SWFH}^G}
\newcommand{\swfz}{\mathit{SWFH}^{\mathbb{Z}/4}}
\newcommand{\undd}{\underline{\delta}}
\newcommand{\ovrr}{\overline{\delta}}
\newcommand{\bv}{\mathcal{V}^+}
\title{A remark on $\mathrm{Pin}(2)$-equivariant Floer homology}
\author{Matthew Stoffregen}
\begin{document}

\begin{abstract}
In this remark, we show how the monopole Fr{\o}yshov invariant, as well as the analogues of the Involutive Heegaard Floer correction terms $\underline{d},\bar{d}$, are related to the $\mathrm{Pin}(2)$-equivariant Floer homology $\mathit{SWFH}^G_*$.  We show that the only interesting correction terms of a $\mathrm{Pin}(2)$-space are those coming from the subgroups $\mathbb{Z}/4$, $S^1$, and $\mathrm{Pin}(2)$ itself.  
\end{abstract} 
\maketitle

\section{Introduction} \label{sec:intro} 
In \cite{ManolescuPin}, Manolescu resolved the triangulation conjecture, establishing that there exist non-triangulable manifolds in all dimensions at least $5$.  The proof relies on the construction of $\mathrm{Pin}(2)$-equivariant Seiberg-Witten Floer homology, where $\mathrm{Pin}(2)$ is the group consisting of two copies of the complex unit circle, with a map $j$ interchanging the two copies and so that $ij=-ji$ and $j^2=-1$.  
 
Let $\mathbb{F}$ denote the field with two elements.   As $S^1$-equivariant monopole Floer homology associates to a three-manifold with $\mathrm{spin}^c$ structure a $H^*(BS^1)=\f[U]$-module, $\mathrm{Pin}(2)$-equivariant Floer homology associates to a rational homology three-sphere with spin structure a $H^*(B\mathrm{Pin}(2))=\f[q,v]/(q^3)$-module, where cohomology is taken with $\mathbb{F}$-coefficients.  From the module structure of $\mathrm{Pin}(2)$-equivariant Floer homology one obtains three invariants of homology cobordism \[\alpha, \beta, \gamma: \ttz \rightarrow \mathbb{Z},\] where $\ttz$ is the integral homology cobordism group of integral homology three-spheres.  These invariants satisfy:
\[ \alpha(Y) \equiv \beta(Y) \equiv \gamma(Y)\equiv \mu(Y) \bmod{2} \]   
and
\[\alpha(-Y) = -\gamma(Y), \; \beta(-Y)=-\beta(Y).\]
In particular, these properties for $\beta$ show that there is no element $Y \in \ttz$ of order $2$ with $\mu(Y)=1$.  Galewski-Stern and Matumoto \cite{GalewskiStern} and \cite{Matumoto} showed that there exist nontriangulable manifolds in all dimensions at least $5$ if and only if there exists an element $Y \in \ttz$ of order $2$ with $\mu(Y)=1$, from which Manolescu's disproof of the triangulation conjecture follows.   

Let $G=\pin$.  Since the introduction of Manolescu's $G$-equvariant Floer homology, denoted $\mathit{SWFH}^G_*(Y,\s)$, other versions of Floer homologies with symmetries beyond the $S^1$-symmetry have become available.  Lin \cite{flin} constructed a $\pin$-equivariant refinement of monopole Floer homology in the setting of Kronheimer-Mrowka \cite{KM}.  In the setting of Heegaard Floer homology introduced by Ozsv{\'a}th-Szab{\'o} in \cite{OzSz1}, \cite{OzSz2}, Hendricks and Manolescu  \cite{HendricksManolescu} point out that naturality questions make it difficult to define a $G$-equivariant version of Heegaard Floer homology.  However, they proceed by considering the subgroup $\mathbb{Z}/4=\langle j \rangle \subset G$ and define a Heegaard Floer analogue of $\mathit{SWFH}^G$ with respect to this smaller group, denoted $\mathit{HFI}(Y,\s)$.  As for the above-mentioned theories, $\hfi(Y,\s)$ is a module over $H^*(B\mathbb{Z}/4)=\f[U,Q]/(Q^2)$.  Using $\mathit{HFI}(Y,\s)$, they associate two homology cobordism invariants $\underline{d}(Y,\s)$, $\overline{d}(Y,\s)$, from the module structure.  However, $\underline{d}$ and $\overline{d}$ do not generally reduce to the Rokhlin invariant mod $2$.  

The purpose of this note is to relate the homology cobordism invariants obtained using theories equivariant with respect to different groups (especially the groups $S^1$, $\mathbb{Z}/4$ and $G$ itself).  In particular, we will see that, roughly speaking, all homology cobordism invariants that are constructed from Manolescu's homotopy type using the Borel homology of a subgroup of $\pin$ are determined by the invariants defined using $\mathbb{Z}/4$, $S^1$ and $G$.  For a precise statement, see Theorem \ref{thm:main}   

We will work in the context of Manolescu's construction of $G$-equivariant Floer homology, that is, $\swfhg$.  

We recall that in order to define $\mathit{SWFH}^G$, Manolescu first associates to a rational homology sphere with spin structure $(Y,\s)$ a $G$-equivariant stable homotopy type, denoted $\swf(Y,\s)$.  Then $\swfhg_*(Y,\s)$ is constructed from $\swf(Y,\s)$ by taking the $G$-equivariant Borel homology
\[\swfhg_*(Y,\s)=\tilde{H}^G_*(\swf(Y,\s)).\]

 Using $\mathbb{Z}/4 =\langle j \rangle \subset G$, we may also consider the $\mathbb{Z}/4$-Borel homology of $\mathit{SWF}(Y,\s)$.  We define 
\[\mathit{SWFH}^{\mathbb{Z}/4}_*(Y,\s)=\tilde{H}^{\mathbb{Z}/4}_*(\mathit{SWF}(Y,\s)).\]
Then $\mathit{SWFH}^{\mathbb{Z}/4}_*(Y,\s)$ has a $H^*(B\mathbb{Z}/4)=\f[U,Q]/(Q^2)$-module structure, from which we will define below homology cobordism invariants $\underline{\delta}(Y,\s)\leq \overline{\delta}(Y,\s)$, which should correspond, respectively, to the invariants $\underline{d}(Y,\s)/2$ and $\overline{d}(Y,\s)/2$ of \cite{HendricksManolescu}.  It is natural to ask to what extent these $\mathbb{Z}/4$ invariants are determined by $\alpha, \beta$ and $\gamma$, and, more generally, by the $\f[q,v]/(q^3)$-module structure of  $\mathit{SWFH}^G_*(Y,\s)$.  We show in Theorem \ref{thm:z4dvsg} how to partially determine $\underline{\delta}(Y,\s)$ and $\overline{\delta}(Y,\s)$ from $\mathit{SWFH}^G_*(Y,\s)$, but that in general $\underline{\delta}(Y,\s)$ and $\overline{\delta}(Y,\s)$ are not determined.  

First, we show that although the $S^1$-Fr{\o}yshov invariant $\delta$ is not determined by $\underline{\delta}$ and $\overline{\delta}$, it is determined by $\swfz$.  
\begin{thm}\label{thm:dvsz4d}
Let $(Y,\s)$ be a rational homology three-sphere with spin structure.  Then 
\[ \delta(Y,\s)=\frac{1}{2} (\min \{ m \equiv 2\mu(Y,\s)+1 \bmod{2} \mid \exists x \in \swfz_m(Y,\s), \; x \in \mathrm{Im} \; U^\ell \; \text{for all} \; \ell \geq 0, x \not \in \mathrm{Im} \; Q\}-1)\]
\end{thm} 
We next relate the $S^1$ and $\mathbb{Z}/4$-invariants with those coming from $G$.  Here, even given $\swfhg_*(Y,\s)$, it is not possible to specify $\delta(Y,\s)$, $\underline{\delta}(Y,\s)$, or $\ovrr(Y,\s)$, although we have the following theorems.    
\begin{thm}\label{thm:dvsg}
Let $(Y,\s)$ be a rational homology three-sphere with spin structure.  Let 
\[ \delta_G(Y,\s)=\frac{1}{2} (\min \{ m \equiv 2\mu(Y,\s)+2 \bmod{4} \mid \exists x \in \swfhg_m(Y,\s), \; x \in \mathrm{Im} \; v^\ell \; \text{for all} \; \ell \geq 0, x \not \in \mathrm{Im} \; q\}-2).\]  Then 
\[ \delta(Y,\s)=\delta_G(Y,\s) \; \text{or} \; \delta_G(Y,\s)+1.\]
\end{thm}
\begin{thm}\label{thm:z4dvsg}
Let $(Y,\s)$ be a rational homology three-sphere with spin structure.  Let 
\[\underline{\delta}_G(Y,\s)=\frac{1}{2} (\min \{ m \equiv 2\mu(Y,\s)+2 \bmod{4} \mid \exists x \in \swfhg_m(Y,\s), \; x \in \mathrm{Im} \; v^\ell \; \text{for all} \; \ell \geq 0, x \not \in \mathrm{Im}\; q^2\}-2),\]
and
\[\overline{\delta}_G(Y,\s)=\frac{1}{2}( \min \{ m \equiv 2\mu(Y,\s)+1 \bmod{4} \mid \exists x \in \swfhg_m(Y,\s), \; x \in \mathrm{Im} \; v^\ell \; \text{for all} \; \ell \geq 0, x \not \in \mathrm{Im}\; q^2\}-1).\]
Then 
\[ \undd(Y,\s)= \underline{\delta}_G(Y,\s) \; \text{or} \; \underline{\delta}_G(Y,\s)+1\] and
\[ \ovrr(Y,\s)= \overline{\delta}_G(Y,\s) \; \text{or} \; \overline{\delta}_G(Y,\s)+1.\]
\end{thm}

To interpret $\delta_G(Y,\s)$, one may think of it just as the invariant $\gamma$, but with an adjustment coming from the $\f[v]$-torsion submodule of $\swfhg_*(Y,\s)$.  Similarly, $\underline{\delta}_G(Y,\s)$ is an adjustment of $\gamma$ as well, while $\overline{\delta}_G(Y,\s)$ is an adjustment of $\beta$.  We have the following as an immediate corollary of Theorem \ref{thm:z4dvsg}.

\begin{cor}\label{cor:ineqs}
Let $(Y,\s)$ be a rational homology three-sphere with spin structure.  Then
\[ \alpha(Y,\s) \geq \overline{\delta}(Y,\s) \geq \beta(Y,\s) \geq \underline{\delta}(Y,\s) \geq \gamma(Y,\s).\]
\end{cor} 

Given that the homology cobordism invariants from $S^1$ and $\mathbb{Z}/4 \subset \mathrm{Pin}(2)$ cannot be determined from the $\mathrm{Pin}(2)$-equivariant homology, it is natural to ask if there are other subgroups of $\mathrm{Pin}(2)$ which produce new homology cobordism invariants.  We show that this is not the case.  We call a homology cobordism invariant $\delta_I$ a generalized Fr{\o}yshov invariant if it is constructed analogously to $\delta$, but perhaps using a different subgroup $H$ of $G$; for the precise definition, see Section \ref{subsec:main}.

\begin{thm}\label{thm:main}
Let $\{\delta_I\}$ be the set of generalized Fr{\o}yshov invariants associated to a subgroup $H \subset G$.  Then \[\{ \delta_I \}\subseteq \{\delta,\underline{\delta},\overline{\delta},\alpha,\beta,\gamma\},\]
where the generalized Fr{\o}yshov invariants are viewed as maps $\ttz \rightarrow \mathbb{Z}$.  
\end{thm}

The proof of Theorem \ref{thm:main} also gives the next corollary.  We note that the closed subgroups of $G$ are precisely $\mathbb{Z}/2=\langle j^2 \rangle$, $\mathbb{Z}/4=\langle j \rangle$, $S^1$, cyclic subgroups of $S^1$, and generalized quaternion groups $Q_{4m}=\langle e^{\pi i/m},j \rangle$.  

\begin{cor}\label{cor:maincor}
Let $H=Q_{4m}\subset G$ for $m$ even and $(Y,\s)$ any rational homology three-sphere with spin structure.  Then the isomorphism type of $\mathit{SWFH}^G_*(Y,\s)$ (as a $H^*(BG)$-module) is specified by the isomorphism type of $\mathit{SWFH}^H_*(Y,\s)$ as a $H^*(BH)$-module.  
\end{cor}

\subsection*{Organization}
In Section \ref{sec:2} we recall what we will need from equivariant topology.  In Section \ref{sec:gysin} we define Gysin sequences and then use these to establish Propositions \ref{prop:z2z4}-\ref{prop:dupdownfromg}, which form the equivariant topology input for Theorem \ref{thm:dvsz4d}-\ref{thm:z4dvsg}.  In Subsection \ref{subsec:gauge} we state the existence of Manolescu's Seiberg-Witten Floer stable homotopy type, $\swf(Y,\s)$, and show Theorems \ref{thm:dvsz4d}-\ref{thm:z4dvsg} of the Introduction.  In Subsection \ref{subsec:main}, we prove Theorem \ref{thm:mainunder}, which is the equivariant topology input for Theorem \ref{thm:main}.  
\section*{Acknowledgments} 
I wish to thank Ciprian Manolescu and Francesco Lin for helpful conversations.  
\section{Spaces of type SWF}\label{sec:2}
 
\subsection{$G$-CW Complexes}
	In this section we recall the definition of spaces of type SWF from \cite{ManolescuPin}, as well as briefly review the basics of equivariant topology, referring to Section 2 of \cite{ManolescuPin} for further details.  Spaces of type SWF are the output of the construction of the Seiberg-Witten Floer stable homotopy type of \cite{ManolescuPin} and \cite{ManolescuK}; see Section \ref{subsec:gauge}.  Throughout, all homology will be taken with $\f=\mathbb{Z}/2$-coefficients.  

Let $K$ be a compact Lie group.  A (finite) $K$-CW decomposition of a space $(X,A)$ with $K$-action is a filtration $(X_n\mid n \in \mathbb{Z}_{\geq 0})$ of $X$ such that 
	\begin{itemize}
	\item $A \subset X_0$ and $X=X_n$ for $n$ sufficiently large.
	\item The space $X_n$ is obtained from $X_{n-1}$ by attaching $K$-equivariant $n$-cells, copies of $(K/H) \times D^n$, for $H$ a closed subgroup of $K$.  
	\end{itemize}
When $A$ is a point, we call $(X,A)$ a pointed $K$-CW complex.  

Let $X$ and $Y$ pointed $K$-CW complexes, at least one of which is a finite complex.  We define the smash product $X \wedge Y$ as a $K$-space by letting $K$ act diagonally.  In the case that $X=V^+$ is the one-point compactification of a finite-dimensional $K$-representation $V$, we call $\Sigma^VY=V^+\wedge Y$ the suspension of $Y$ by $V$.  
Define also \[X\wedge_K Y =(X\wedge Y)/K.\]

Let $EK$ denote a contractible space with free $K$-action, and let $EK_+$ denote $EK$ with a disjoint base point added.  The reduced Borel homology and cohomology of $X$ are defined by: 
\begin{align}\label{eq:boreldef}
\tilde{H}^K_*(X)&= \tilde{H}_*(EK_+ \wedge_K X),\\ \nonumber
\tilde{H}^*_K(X)&= \tilde{H}^*(EK_+\wedge_K X).
\end{align}
Borel homology and cohomology are invariants of $K$-equivariant homotopy equivalence.

Furthermore, there is a projection:
\[ \begin{tikzcd} EK_+\wedge_K X \arrow{r}{\pi_{1}} & EK_+/K=BK_+,\end{tikzcd}\]
which induces a map $\pi_1^*:H^*(BK)\rightarrow \tq_K(X)$.  Via $\pi_1^*$, $\tq_K(X)$ and $\tilde{H}^K_*(X)$ inherit the structure of $H^*(BK)$-modules.  On $\tq_K(X)$, $H^*(BK)$ increases grading, while on $\tilde{H}^K_*(X)$, $H^*(BK)$ decreases grading.

For a subgroup $L\subseteq K$ and a $K$-CW complex $X$, we may relate the $L$-Borel cohomology and $K$-Borel cohomology of $X$.  Indeed, there is a quotient map \[\begin{tikzcd}
EK_+\wedge_L X  \arrow{r}{\pi} & EK_+ \wedge_K X \end{tikzcd}
\]
which induces a map in cohomology:
\begin{equation}\label{eq:equiquot}\begin{tikzcd}
\tq_K(X) \arrow{r}{\pi^*} & \tq_L(X).\end{tikzcd}
\end{equation}
We call the map $\pi^*$ the \emph{restriction map} from $K$ to $L$ and write $\pi^*=\mathrm{res}^K_L$.   

Setting $X$ a point in (\ref{eq:equiquot}), we have a restriction map (of algebras) $H^*(BK) \to H^*(BL)$.  Then $\tq_L(X)$ inherits an $H^*(BK)$-module structure by restriction.  We denote $\tq_L(X)$, viewed as a $H^*(BK)$-module, by $\mathrm{res}^K_L\tq_L(X)$.  Then the map 
\begin{equation}\label{eq:equiquot2}\begin{tikzcd}\tq_K(X) \arrow{r}{\mathrm{res}^K_L} & \mathrm{res}^K_L\tq_L(X) \end{tikzcd}\end{equation}
is a map of $H^*(BK)$-modules.    

Let $G=\mathrm{Pin}(2)$ and $BG$ its classifying space.  In addition to the definition of $G$ from the Introduction, one may think of $G$ as the set $S^1 \cup jS^1 \subset \mathbb{H}$, where $S^1$ is the unit circle in the $\langle 1, i \rangle$ plane, with group action on $G$ induced from the group action of the unit quaternions.  Thus $S^\infty=S(\mathbb{H}^\infty)$ with its quaternion action is a free $G$-space.  Since $S^\infty$ is contractible, we identify $EG=S^\infty$.  

Manolescu shows in \cite{ManolescuPin} that \begin{equation}\label{eq:bghomcalc}H^*(BG)=\f[q,v]/(q^3),\end{equation} where $\mathrm{deg} \; q =1$ and $\mathrm{deg} \; v =4$.  

For convenience, we also record 
\begin{align}\label{eq:classcohcalc}H^*(B\mathbb{Z}/2)&=\f[W],\\ \nonumber H^*(B\mathbb{Z}/4)&=\f[U,Q]/(Q^2=0),\\ \nonumber H^*(BS^1)=H^*(\mathbb{C}P^\infty)&=\f[U],\end{align} where $\mathrm{deg}\, U=2$ and $\mathrm{deg}\, W=\mathrm{deg}\, Q=1$. 

For a subset $S$ of a group $K$, let $\langle S \rangle$ denote the subgroup generated by $S$.  There are inclusions $\mathbb{Z}/2=\langle j^2 \rangle \subset S^1 \subset G$, and $\mathbb{Z}/4 \cong \langle j \rangle \subset G$.  We will describe the corresponding restriction maps in Proposition \ref{prop:listofgys}. 

We will also need to use that Borel cohomology behaves well with respect to suspension.
\begin{prop}[\cite{ManolescuPin} Proposition 2.2]\label{prop:susp}
Let $V$ a finite-dimensional representation of a compact Lie group $K$.  Then, as $H^*(BK)$-modules:
\begin{align}\label{eq:suspensioninvar}
\tilde{H}^*_K(\Sigma^VX)\cong \tilde{H}^{*-\mathrm{dim}\, V}_K(X)\\ \nonumber
\tilde{H}_*^K(\Sigma^VX)\cong \tilde{H}_{*-\mathrm{dim}\, V}^K(X) 
\end{align} 
\end{prop}

We mention three irreducible representations of $G$:
\begin{itemize}
\item The trivial one-dimensional representation $\mathbb{R}$.
\item The one-dimensional real vector space on which $j\in G$ acts by $-1$, and on which $S^1$ acts trivially, denoted $\tilde{\mathbb{R}}$.
\item The quaternionic representation $\mathbb{H}$, where $G$ acts by left multiplication.
\end{itemize}
\begin{defn}\label{thm:swfdefn}
Let $s \in \mathbb{Z}$.  A \emph{space of type SWF} at level $s$ is a pointed, finite $G$-CW complex $X$ with
\begin{itemize}
\item The $S^1$-fixed-point set $X^{S^1}$ is $G$-homotopy equivalent to $(\tilde{\mathbb{R}}^s)^+$, the one-point compactification of $\tilde{\mathbb{R}}^s$.
\item The action of $G$ on $X-X^{S^1}$ is free.  
\end{itemize}
We define $\mu(X)\in \mathbb{Q}/2\mathbb{Z}$ by $\mu(X)=\frac{s}{2} \bmod{2}$.   
\end{defn}	

We will often have occasion later to work with $2\mu(X)$, which we view as an element of $\mathbb{Q}/4\mathbb{Z}$.  

We note that for a space $X$ of type SWF, \[\tilde{H}^G_*(X^{S^1})=\tilde{H}^G_*((\tilde{\mathbb{R}}^s)^+)=\tilde{H}^G_{*-s}(S^0)=H_{*-s}(BG),\] and
\[\tilde{H}^*_G(X^{S^1})=H^{*-s}(BG),\]
using Proposition \ref{prop:susp}.

Associated to a space $X$ of type SWF at level $s$, we take the Borel cohomology $\tilde{H}_G^*(X)$, from which Manolescu \cite{ManolescuPin} defines $a(X),b(X),$ and $c(X)$:
\begin{align}\label{eq:adef} 
a(X) &= \mathrm{min}\{ r \equiv 2\mu(X) \;\mathrm{mod}\; 4 \mid \exists \, x \in \tilde{H}^r_G(X), v^lx \neq 0 \; \mathrm{ for\; all} \; l \geq 0 \}, \\ \nonumber
b(X) &= \mathrm{min}\{ r \equiv 2\mu(X)+1 \;\mathrm{mod}\; 4 \mid \exists \, x \in \tilde{H}^r_G(X), v^lx \neq 0 \; \mathrm{for\; all}\; l \geq 0 \}-1,  
\\ \nonumber
c(X) &= \mathrm{min}\{ r \equiv 2\mu(X)+2 \;\mathrm{mod}\; 4 \mid \exists \, x \in \tilde{H}^r_G(X), v^lx \neq 0 \;\mathrm{for\; all}\; l \geq 0 \}-2. 
\end{align}
Using $S^1$-Borel cohomology, Manolescu \cite{ManolescuPin} also defines
\begin{equation}\label{eq:ddef}
d(X) = \mathrm{min}\{ r  \mid \exists \, x \in \tilde{H}^r_{S^1}(X), U^lx \neq 0 \; \mathrm{ for\; all} \; l \geq 0 \}.
\end{equation}
The well-definedness of $a,b, c$, and $d$ follows from the Equivariant Localization Theorem.  We list a version of this theorem for spaces of type SWF:

\begin{thm}[\cite{tomdieck} III (3.8)]\label{thm:equivlocalzn}
Let $X$ be a space of type SWF.  Then the inclusion $X^{S^1}\rightarrow X$, after inverting $v$, induces an isomorphism of $\f[q,v,v^{-1}]/(q^3)$-modules:
\[v^{-1}\tilde{H}^*_G(X^{S^1})\cong v^{-1}\tilde{H}^*_G(X).\]
Furthermore,
\begin{align}\nonumber U^{-1}\tq_{S^1}(X^{S^1}) &\cong U^{-1}\tq_{S^1}(X),\\ \nonumber U^{-1}\tq_{\mathbb{Z}/4}(X^{S^1})&\cong U^{-1} \tq_{\mathbb{Z}/4}(X),\\\nonumber W^{-1}\tq_{\mathbb{Z}/2}(X^{S^1})&\cong W^{-1} \tq_{\mathbb{Z}/2}(X).\end{align}
\end{thm}
For $X$ a space of type SWF, $X$ is a finite $G$-complex and so we have that $\tilde{H}^*_G(X)$ is finitely generated as an $\f[v]$-module.  In particular, the $\f[v]$-torsion part of $\tilde{H}^*_G(X)$ is bounded above in grading.  Similarly, the $\f[U]$-torsion parts of $\tq_{S^1}(X)$ and $\tq_{\mathbb{Z}/4}(X)$, as well as the $\f[W]$-torsion of $\tq_{\mathbb{Z}/2}(X)$, are bounded above in grading.  

Following (\ref{eq:adef}), we define analogues of $a,b$ and $c$ for $\mathbb{Z}/4$.  

\begin{defn}\label{def:prenewdinvars}
For $X$ a space of type SWF, we define $\overline{d}(X)$ and $\underline{d}(X)$ by
\begin{align}
\overline{d}(X) &= \mathrm{min}\{ r \equiv 2\mu(X) \;\mathrm{mod}\; 2 \mid \exists \, x \in \tilde{H}^r_{\mathbb{Z}/2}(X), U^lx \neq 0 \; \mathrm{ for\; all} \; l \geq 0 \}, \\ \nonumber
\underline{d}(X) &= \mathrm{min}\{ r \equiv 2\mu(X)+1 \;\mathrm{mod}\; 2 \mid \exists \, x \in \tilde{H}^r_{\mathbb{Z}/2}(X), U^lx \neq 0 \; \mathrm{for\; all}\; l \geq 0 \}-1,
\end{align}
The well-definedness of $\overline{d}(X)$ and $\underline{d}(X)$ follows from Theorem \ref{thm:equivlocalzn}.  
\end{defn} 

\subsection{Stable $G$-Equivariant Topology}
Here we define stable equivalence for $G$-spaces and define the Manolescu invariants $\alpha, \beta$ and $\gamma$, as well as their $\mathbb{Z}/4$-analogues.
\begin{defn}[see \cite{ManolescuK}]
Let $X$ and $X'$ be spaces of type SWF, $m,m' \in \mathbb{Z}$, and $n,n' \in \mathbb{Q}$.  We say that the triples $(X,m,n)$ and $(X',m',n')$ are \emph{stably equivalent} if $n-n' \in \mathbb{Z}$ and there exists a $G$-equivariant homotopy equivalence, for some $r \gg 0$ and some nonnegative $M \in \mathbb{Z}$ and $N \in \mathbb{Q}$:
\begin{equation}\label{eq:gstabdef} \Sigma^{r\mathbb{R}} \Sigma^{(M-m)\tilde{\mathbb{R}}}\Sigma^{(N-n)\mathbb{H}}X \rightarrow \Sigma^{r\mathbb{R}} \Sigma^{(M-m')\tilde{\mathbb{R}}}\Sigma^{(N-n')\mathbb{H}}X'.\end{equation}
\end{defn} 
Let $\mathfrak{E}$ be the set of equivalence classes of triples $(X,m,n)$ for $X$ a space of type $\mathrm{SWF}$, $m\in \mathbb{Z}$, $n \in \mathbb{Q}$, under the equivalence relation of stable $G$-equivalence\footnote{This convention is slightly different from that of \cite{ManolescuK}.  The object $(X,m,n)$ in the set of stable equivalence classes $\mathfrak{E}$, as defined above, corresponds to $(X,\frac{m}{2},n)$ in the conventions of \cite{ManolescuK}.}.  The set $\mathfrak{E}$ may be considered as a subcategory of the $G$-equivariant Spanier-Whitehead category \cite{ManolescuPin}, by viewing $(X,m,n)$ as the formal desuspension of $X$ by $m$ copies of $\tilde{\mathbb{R}} $ and $n$ copies of $\mathbb{H}$.  We define Borel cohomology for $(X,m,n) \in \E$, as an isomorphism class of $H^*(BK)$-modules, by 
\begin{align}\label{eq:borelgeneral}
\tq_K((X,m,n))&=\tq_K(X)[m+4n],
\end{align}
for $K$ any closed subgroup of $G$.  The well-definedness of (\ref{eq:borelgeneral}) follows from Proposition \ref{prop:susp}.

Finally, we define the invariants $\alpha,\beta, \gamma$, $\delta$, $\underline{\delta}$ and $\overline{\delta}$ associated to an element of $\mathfrak{E}$.  
\begin{defn}\label{def:manolescudefn1}
For $[(X,m,n)] \in \E$, we set 

\begin{equation*} 
\alpha((X,m,n))=\frac{a(X)}{2}-\frac{m}{2}-2n,\;
\beta((X,m,n))=\frac{b(X)}{2}-\frac{m}{2}-2n,\;
\gamma((X,m,n))=\frac{c(X)}{2}-\frac{m}{2}-2n,
\end{equation*}
\begin{equation*}
\delta((X,m,n))=\frac{d(X)}{2}-\frac{m}{2}-2n,
\end{equation*}
\begin{equation*}
\overline{\delta}((X,m,n))=\frac{\overline{d}(X)}{2}-\frac{m}{2}-2n,\;
\underline{\delta}((X,m,n))=\frac{\underline{d}(X)}{2}-\frac{m}{2}-2n.
\end{equation*}
The invariants above do not depend on the choice of representative of the class $[(X,m,n)]\in \mathfrak{E}$.  
\end{defn} 

For notational convenience later, we also define 
\[\mu((X,m,n))=\mu(X)-\frac{m}{2}-2n \bmod{2}.\]
 
\section{Gysin Sequences}\label{sec:gysin}
\subsection{Gysin Sequences for Change-of-Groups}
Let $K$ and $L$ be compact Lie groups so that there is a fiber sequence of Lie groups (for $n=0,1,$ or $3$),
\begin{equation}\label{eq:liebund}L \subset K\rightarrow S^n.\end{equation}
Then there is a Gysin sequence, as in \cite{tomdieck}[\S III.2], given by: 
\begin{equation}\label{eq:liegys}\begin{tikzcd} \dots \arrow{r}& H^*(BK) \arrow{r}{ e(K,L)\cup -} &H^{*+n+1}(BK)\arrow{r}{p^*} &H^{*+n+1}(BL) \arrow{r}& H^{*+1}(BK)\arrow{r}&\dots
\end{tikzcd} \end{equation}
where $e(K,L)$ is the Euler class of the sphere bundle $S^n \rightarrow BL=EK/L\rightarrow BK$ and $p:BL \rightarrow BK$ is the projection (note that $p^*=\mathrm{res}^K_L$).  

\begin{prop}\label{prop:listofgys} 
We specify the Euler classes associated to the groups $\mathbb{Z}/2,\mathbb{Z}/4,S^1,$ and $G$, as follows.
\begin{enumerate}
\item \label{itm:typ1} Associated to $S^1\rightarrow G \rightarrow S^0$, we have $e(G,S^1)=q$.  Further, $p^*v=U^2$.

\item \label{itm:typ2} Associated to $\mathbb{Z}/4 \rightarrow G \rightarrow S^1,$ we have $e(G,\mathbb{Z}/4)=q^2$.  Further, $p^*q=Q$ and $p^*v=U^2$.

\item \label{itm:typ3} Associated to $ \mathbb{Z}/2=\langle j^2 \rangle \rightarrow \mathbb{Z}/4 \rightarrow \langle j \rangle/\langle j^2\rangle =\mathbb{Z}/2,$ we have $e(\mathbb{Z}/4,\mathbb{Z}/2)=Q$.  Further, $p^*Q=0$, $p^*U=W^2$.

\item \label{itm:typ4} Associated to $\mathbb{Z}/2=\langle j^2 \rangle \rightarrow S^1 \rightarrow S^1/\langle j^2 \rangle =S^1
,$ we have $e(S^1,\mathbb{Z}/2)=0$.  Further, $p^*U=W^2$.
\end{enumerate}

We call the Gysin sequences above Types (\ref{itm:typ1})-(\ref{itm:typ4}), respectively.
\end{prop}
\begin{proof}
In each case (\ref{itm:typ1})-(\ref{itm:typ4}) it is straightforward to see that the Euler class is specified by the algebraic structure of the entries of the exact sequence (\ref{eq:liegys}).  For example, we prove (\ref{itm:typ1}).  Since $H^1(BS^1)=0$, we have that $p^*q=0$.  By exactness of (\ref{eq:liegys}), $q$ is in the image of the Euler class, and since the Euler class $e(G,S^1)$ is of degree $1$, we have $e(G,S^1)=q$.  The other cases are similar.    
\end{proof}

More generally, for $X$ a $K$-CW complex, we have a sphere bundle:
\begin{equation}\label{eq:liebund2}S^n \rightarrow EK \times_L X \rightarrow EK \times_K X\end{equation}
and a Gysin sequence:
\begin{equation}\label{eq:gysin1}\begin{tikzcd} H^*_K(X) \arrow{r}{e(X)\cup -}& H^{*+n+1}_K(X) \arrow{r}{p^*} & H^{*+n+1}_L(X)\arrow{r} & \dots, \end{tikzcd}\end{equation}
where $e(X)$ is the Euler class of the bundle (\ref{eq:liebund2}).  By construction, we have a map of bundles:
\[\begin{tikzcd}
EK \times_L X \arrow{r}\arrow{d}{\pi_L} & EK\times_K X \arrow{d}{\pi_K} \\
BL \arrow{r} & BK.
\end{tikzcd}\]
and by functoriality of the Euler class, we have
\begin{equation}\label{eq:eulerfunc} 
e(X) =\pi^*_K(e(K,L)).
\end{equation}
\begin{fact}\label{fct:eulerclasses}
By (\ref{eq:eulerfunc}), $e(X)=q,q^2,Q,0$ for types (\ref{itm:typ1})-(\ref{itm:typ4}), respectively, for any $K$-CW complex $X$.  
\end{fact}

We can now relate, in the case of spaces of type SWF, the $\mathbb{Z}/2,\mathbb{Z}/4,S^1,$ and $G$-cohomology theories.  

We adapt a definition of \cite{flin2} to our setting.
\begin{defn}\label{def:absgys}
Let $\mathcal{S}=(L \rightarrow K \rightarrow S^n)$ be one of the sequences of groups in Proposition \ref{prop:listofgys}.  An \emph{abstract $\mathcal{S}$-Gysin sequence} $\mathcal{G}$ consists of the following:
\begin{enumerate}
\item\label{itm:abs1} A $H^*(BK)$-module $M^K$, a $H^*(BL)$-module $M^L$, both graded by a $\mathbb{Z}$-coset of $\mathbb{Q}$ and bounded below.  
\item\label{itm:abs2} An exact triangle of $H^*(BK)$-modules.
\begin{equation}\label{eq:gysintriangleabs}
\begin{tikzcd}
M^K \arrow{rr}{e(K,L)} & &M^K\arrow{ld}{p^*} \\
&  \mathrm{res}^K_L M^L\arrow{lu}{\iota^*} & 
\end{tikzcd}
\end{equation}
where $e(K,L)$ is the Euler class of $H^*(BK)$ as in Proposition \ref{prop:listofgys}, acting on the $H^*(BK)$-module $M^K$.  Further, $\iota^*$ has degree $-n$ and $p^*$ has degree $0$.  
\item\label{itm:abs3} In sufficiently high degrees, the triangle (\ref{eq:gysintriangleabs}) is isomorphic to the exact triangle corresponding to $\mathcal{S}$ from Proposition \ref{prop:listofgys}, perhaps with grading shifted.  
\end{enumerate}

\end{defn}
\begin{prop}\label{prop:exiofabsgys} 
For every $X \in \E$, and $\mathcal{S}=(L \rightarrow K \rightarrow S^n)$ the Gysin sequence 
\begin{equation}\label{eq:putativegysin}\begin{tikzcd}
\tq_K(X) \arrow{rr}{e(K,L)} & & \tq_K(X)\arrow{ld}{p^*} \\
&  \tq_L(X)\arrow{lu}{\iota^*} & 
\end{tikzcd}\end{equation}
is an abstract $\mathcal{S}$-Gysin sequence, for $\mathcal{S}$ of type (\ref{itm:typ1})-(\ref{itm:typ4}).  The grading shift in (\ref{itm:abs3}) is $2\mu(X)$, upward.   
\end{prop}
\begin{proof}
Properties (\ref{itm:abs1}) and (\ref{itm:abs2}) of Definition \ref{def:absgys} are automatically satisfied for (\ref{eq:putativegysin}); we prove Property (\ref{itm:abs3}).  In sufficiently high degrees $d\geq N$ for some $N$, $\tu^d_K(X)$ must be isomorphic to $H^{d+2\mu(X)}(BK)$ and $\tu^d_L(X)$ must be isomorphic to $H^{d+2\mu(X)}(BL)$, using that $X$ is of type SWF.  Recall from the proof of Proposition \ref{prop:listofgys} that there is only one choice of maps $p^*$ and $\iota^*$ that make the triples $H^*(BK),H^*(BK)$ and $H^*(BL)$ into exact triangles.  Since $\tq_K(X)$ and $\tq_L(X)$ are isomorphic to $H^{*+2\mu(X)}(BK)$ and $H^{*+2\mu(X)}(BL)$, the same reasoning as in the proof of Proposition \ref{prop:listofgys} shows there is only one choice of maps $p^*$ and $\iota^*$ that make (\ref{eq:putativegysin}) exact in high degrees (Namely, the $p^*$ and $\iota^*$ listed).  This establishes property (\ref{itm:abs3}) of Definition \ref{def:absgys}. 
\end{proof} 

In order to prove Proposition \ref{prop:z2z4}, the precursor to Theorem \ref{thm:dvsz4d}, we will need to compare $S^1$ and $\mathbb{Z}/4$-invariants despite there being no Gysin sequence relating $S^1$ and $\mathbb{Z}/4$ homology.  As an intermediate step, we define
\begin{equation}\label{eq:defofdzt}
\delta_{\mathbb{Z}/2}(X)=\frac{1}{2}(\mathrm{min}\{ m\mid \exists x \in H^m_{\mathbb{Z}/2}(X), W^\ell x\neq 0 \; \mathrm{for\; all}\; \ell \geq 0\})
\end{equation} 
for $X \in \E$.  

A priori, $\dzt(X)-\mu(X)$ may be a half-integer (the next lemma implies it is, in fact, an integer). 
\begin{lem}\label{lem:z2s1}
Let $X \in \E$.  Then $\delta_{\mathbb{Z}/2}(X)=\delta(X)$.
\end{lem}
\begin{proof}
We will use the Gysin sequence of type (\ref{itm:typ4}) associated to $X$.  For now, fix $p^*$ and $\iota^*$ to refer to the Gysin sequence maps of that type.  
 
First, we establish $\delta(X)\geq\delta_{\mathbb{Z}/2}(X)$.  Let $x \in \tu^m_{S^1}(X)$ be $U$-nontorsion.  For sufficiently large $\ell$, by Property (\ref{itm:abs3}) of Definition \ref{def:absgys}, $p^*(U^{\ell} x)\neq 0$.  By the equivariance property (\ref{itm:abs2}) of the same definition, $p^*(U^{\ell} x)=W^{2\ell}p^*x$, and so $p^*x$ is a $W$-nontorsion element of $\tu^m_{\mathbb{Z}/2}(X)$.  For notational convenience, define 
\[ \underline{d}_{\mathbb{Z}/2}(X) = \min\{ m\equiv 2\mu(X) \bmod{2} \mid \exists x \in H^m_{\mathbb{Z}/2}(X), W^\ell x\neq 0 \; \text{for all}\; \ell \geq 0\}.\]We have then
\begin{equation}\label{eq:prop3eq1} \min \{ m \equiv 2\mu(X) \bmod{2} \mid \exists x \in H^m_{S^1}(X), \; U^\ell x \neq 0 \;\text{for all}\; \ell \geq 0 \} \geq \underline{d}_{\mathbb{Z}/2}(X).\end{equation}
We note that the only $U$-nontorsion elements of $\tq_{S^1}(X)$ are in degree $d \equiv 2\mu(X) \bmod{2}$, by Proposition \ref{prop:exiofabsgys}.  So the left-hand side of (\ref{eq:prop3eq1}) is $2\delta(X)$.  We also define
\[\overline{d}_{\mathbb{Z}/2}(X)= \min\{ m\equiv 2\mu(X)+1 \bmod{2} \mid \exists x \in H^m_{\mathbb{Z}/2}(X), W^\ell x\neq 0 \; \mathrm{for\; all}\; \ell \geq 0\}.\]  By definition, $\delta_{\mathbb{Z}/2}(X)=\frac{\min\{\underline{d}_{\mathbb{Z}/2}(X),\; \overline{d}_{\mathbb{Z}/2}(X)\}}{2}$. 

We next show the inequality opposite to (\ref{eq:prop3eq1}).  

Let $x\in \tu^m_{\mathbb{Z}/2}(X)$ be a $W$-nontorsion element, with $m \equiv 2\mu(X)+1 \bmod{2}$.  Then, by Property (\ref{itm:abs3}) of Definition \ref{def:absgys}, $\iota^*(W^{2\ell}x)=U^{\ell}\iota^*x$ must be nonzero.  In particular, $\iota^*x\in \tu^{m-1}_{S^1}(X)$ is $U$-nontorsion.  Then we obtain 
\begin{equation}\label{eq:prop3eq2} 2\delta(X)\leq \overline{d}_{\mathbb{Z}/2}(X)-1.\end{equation}
Furthermore,
\begin{equation}\label{eq:prop3eq3}
\overline{d}_{\mathbb{Z}/2}(X)=\underline{d}_{\mathbb{Z}/2}(X)\pm 1,
\end{equation}
since for $x$ in one of the sets corresponding to $\underline{d}_{\mathbb{Z}/2}$ or $\overline{d}_{\mathbb{Z}/2}$, $Wx$ is in the other.  

But, combining (\ref{eq:prop3eq1}) and (\ref{eq:prop3eq2}), we have $\underline{d}_{\mathbb{Z}/2}(X)\leq \overline{d}_{\mathbb{Z}/2}(X)-1$.  So, from (\ref{eq:prop3eq3}), we obtain $\underline{d}_{\mathbb{Z}/2}(X)= \overline{d}_{\mathbb{Z}/2}(X)-1$.  It follows that $2\delta(X)=\underline{d}_{\mathbb{Z}/2}(X)$.  

Using the definition of $\delta_{\mathbb{Z}/2}(X)$, the proof is complete.  \end{proof}

The following statement corresponds to Theorem \ref{thm:dvsz4d} of the Introduction.  
\begin{prop}\label{prop:z2z4}
Let $X\in \E$.  Then:
\begin{equation}\label{eq:propz2z4} \delta(X)=\frac{1}{2}(\mathrm{min}\{ m\equiv 2\mu(X)+1 \bmod{2} \mid \exists x \in H^m_{\mathbb{Z}/4}(X),\; U^\ell x\neq 0 \; \mathrm{for \; all}\; \ell\geq 0, \; Qx=0\}-1).\end{equation}
\end{prop}
\begin{proof}
We denote the right-hand side of (\ref{eq:propz2z4}) by $\delta_{\mathbb{Z}/4}(X)$.  We will consider the abstract Gysin sequence of type (\ref{itm:typ3}) associated to $X$; fix $p^*$ and $\iota^*$ to refer to the maps in this type of Gysin sequence.  Using Lemma \ref{lem:z2s1}, we need only show \begin{equation}\label{eq:prop3dztvsz4} \dzt(X)=\delta_{\mathbb{Z}/4}(X).\end{equation}
We start by showing $\dztzf(X)\leq \dzt(X)$.  We note that any $W$-nontorsion element $x$ of $\tq_{\mathbb{Z}/2}(X)$ with $\mathrm{deg}\; x \equiv 2\mu(X)+1 \bmod{2}$ must have $U^\ell\iota^* x=\iota^* W^{2\ell}x\neq 0$ for $\ell$ sufficiently large.  However, $Q\iota^* x =0$ by exactness of (\ref{eq:gysintriangleabs}).  Thus, if $x \in \tu^m_{\mathbb{Z}/2}(X)$ with $m \equiv 2\mu(X)+1\; \bmod{2}$ is $W$-nontorsion, then there exists an element $\iota^* x \in \tu^m_{\mathbb{Z}/4}(X)$ which is $U$-nontorsion, and which is annihilated by $Q$.  Thus
\[ \min \{ m \equiv 2\mu(X)+1 \bmod{2} \mid \exists x \in \tu^m_{\mathbb{Z}/4}(X), \; U^\ell x \neq 0 \; \mathrm{for\; all}\; \ell \geq 0, \; Qx=0\}-1\leq \overline{d}_{\mathbb{Z}/2}(X)-1.\]
By the proof of Lemma \ref{lem:z2s1}, $\frac{\overline{d}_{\mathbb{Z}/2}(X)-1}{2}=\delta_{\mathbb{Z}/2}(X)$. Then 
\[ \dztzf(X)\leq \dzt(X).\] 

Next we show \begin{equation}\label{eq:prop3last}\dzt(X) \leq \dztzf(X).\end{equation}  Indeed, fix $m \equiv 2\mu(X)+1 \; \bmod{2}$ and $x \in \tu^m_{\mathbb{Z}/4}(X)$ so that $x$ is $U$-nontorsion and satisfies $Qx=0$.  Then $x \in \mathrm{Im} \;\iota$, say $x=\iota y$, by exactness of (\ref{eq:gysintriangleabs}).  However, since $x$ is $U$-nontorsion, $y$ is nontorsion as well, so we see:
\[\min\{ m\mid \exists x \in \tu^m_{\mathbb{Z}/2}(X), W^{\ell} x \neq 0 \; \text{for all }\; \ell \geq 0\} \leq 2\delta_{\mathbb{Z}/4}(X).\]
Recalling the definition of $\dzt(X)$, the above is precisely $\dzt(X)\leq \dztzf(X)$, completing the proof.  
\end{proof} 

The following proposition corresponds to Theorem \ref{thm:dvsg} from the Introduction.
\begin{prop}\label{prop:dfromg}
Let $X\in \E$.  Fix $N\in\mathbb{Z}$.  Then $\delta(X) \leq 2N+\mu(X)+1$ if and only if 
\begin{equation}\label{eq:espacesdfromg}
\frac{\min\{ m \equiv 2\mu(X)+2 \bmod{4} \mid \exists x \in \tu^m_G(X), \; v^\ell x \neq 0 \; \text{for all}\; \ell \geq 0, \; qx=0\}-2}{2} \leq 2N+\mu(X).
\end{equation}
\end{prop}
\begin{proof}
Denote the left-hand side of (\ref{eq:espacesdfromg}) by $\dg(X)$.  First, we show that $\dg(X) \leq 2N+\mu(X)$ implies $\delta(X) \leq 2N+\mu(X)+1$.  

Say $\dg(X)\leq 2N+\mu(X)$, that is, there exists a $v$-nontorsion element $x \in \tu^{4N+2\mu(X)+2}_G(X)$ so that $qx=0$.  Then, by exactness of (\ref{eq:gysintriangleabs}), $x \in \mathrm{Im}\; \iota^*$, say $x=\iota^* y$.   Since $x$ is $v$-nontorsion, $y$ must be $U$-nontorsion.  Thus, $\delta(X)\leq \frac{\mathrm{deg} y}{2}=2N+\mu(X)+1$.  

Next, say that $\delta(X) \leq 2N+\mu(X)+1$.  Let $x \in \tu^{4N+2\mu(X)+2}_{S^1}(X)$ be $U$-nontorsion.  Then by (\ref{itm:typ1}) of Proposition \ref{prop:listofgys}, $\iota^* x $ must be $v$-nontorsion.  In particular, $\iota^* x$ is a $v$-nontorsion element of $\tu^{4N+2\mu(X)+2}_G(X)$ with $q(\iota^* x)=0$.  Thus  $\dg(X)\leq 2N+\mu(X)$, as needed.  
\end{proof}  

The following proposition corresponds to Theorem \ref{thm:z4dvsg} from the Introduction.
\begin{prop}\label{prop:dupdownfromg}
Let $X\in \E$.  Then $\underline{\delta}(X)\leq 2N+\mu(X)+1$, for some integer $N$, if and only if 
\begin{equation}\label{eq:dlowere}
\frac{\min\{ m \equiv 2\mu(X)+2 \; \bmod{4} \mid\exists x \in \tu^m_G(X), \; v^\ell x \neq 0 \; \text{for all}\; \ell \geq 0, \; q^2x=0\}-2}{2} \leq 2N+\mu(X).
\end{equation}
Further, $\overline{\delta}(X) \leq 2N+\mu(X)+1$ if and only if
\begin{equation}\label{eq:duppere}
\frac{\min\{ m\equiv 2\mu(X)+1\; \bmod{4} \mid \exists x \in \tu^m_G(X), \; v^\ell x \neq 0 \; \text{for all}\; \ell \geq 0, \; q^2x=0\}-1}{2} \leq 2N+\mu(X).
\end{equation}
\end{prop}
\begin{proof}  
We will consider Gysin sequences of Type \ref{itm:typ2}.  Denote the left-hand side of (\ref{eq:dlowere}) by $\dglow(X)$, and that of (\ref{eq:duppere}) by $\dgup(X)$.  

First, we show that $\underline{\delta}(X)\leq 2N+\mu(X)+1$ implies $\dglow(X)\leq 2N+\mu(X)$.  Indeed, say $\underline{\delta}(X) \leq 2N+\mu(X)+1$ and $x\in \tu^{4N+2\mu(X)+3}_{\mathbb{Z}/4}(X)$ so that $x$ is $U$-nontorsion.  Then by (\ref{itm:typ2}) of Proposition \ref{prop:listofgys}, $\iota x$ is $v$-nontorsion.  Further, $q^2\iota x=0$ by exactness of (\ref{eq:gysintriangleabs}).  Thus
\[\min\{ m \equiv 2\mu(X)+2 \; \bmod{4} \mid\exists x \in \tu^m_G(X), \; v^\ell x \neq 0 \; \text{for all}\; \ell \geq 0, \; q^2x=0\}\leq 4N+2\mu(X)+2.\]
Then $\dglow(X)\leq 2N+\mu(X)$, as needed.  

Next, suppose $\dglow(X)\leq 2N+\mu(X)$; we will show $\underline{\delta}(X)\leq 2N+\mu(X)+1$.  Choose $x \in \tu_G^{4N+2\mu(X)+2}(X)$ so that $x$ is $v$-nontorsion and $q^2x = 0$.  Then $x \in \mathrm{Im} \;\iota$, say $x=\iota y$, and $y$ is $U$-nontorsion, in grading $4N+2\mu(X)+3$.  We then obtain $\underline{\delta}(X)\leq 2N+\mu(X)+1$, as needed.

The proof for $\overline{\delta}(X)$ is completely analogous.  
\end{proof}

Propositions \ref{prop:dfromg} and \ref{prop:dupdownfromg} cannot be improved, as we will see in Example \ref{ex:weirdgkoszulexamples}, in that there exist spaces $X_1,X_2$ with isomorphic $\tq_G(X_i)$ but different $\delta$, $\underline{\delta}$ and $\overline{\delta}$ invariants.   
\subsection{Seiberg-Witten Floer Homology}\label{subsec:gauge}
In this section we convert the results of the previous section into statements for three-manifolds.  First we recall the existence of the Seiberg-Witten Floer stable homotopy type.

\begin{thm}[Manolescu {\cite{ManolescuPin},\cite{ManolescuK}}]\label{thm:ManolescuSWF}
There is an invariant $\mathit{SWF}(Y,\mathfrak{s})$, the Seiberg-Witten Floer spectrum class, of rational homology three-spheres with spin structure $(Y,\s)$, taking values in $\E$.  A spin cobordism $(W,\mathfrak{t})$, with $b_2(W)=0$, from $Y_1$ to $Y_2$, induces a map $\mathit{SWF}(Y_1, \mathfrak{t}|_{Y_1}) \rightarrow \mathit{SWF}(Y_2,\mathfrak{t}|_{Y_2})$. The induced map is a homotopy-equivalence on $S^1$-fixed-point sets.  
\end{thm}

Manolescu constructs $\mathit{SWF}(Y,\s)$ by using finite-dimensional approximation to the Seiberg-Witten equations on the Coulomb slice.  From $\mathit{SWF}(Y,\s)$ one extracts homology cobordism invariants as in the following definition.

\begin{defn}\label{def:mano3invars} For $(Y,\s)$ a spin rational homology three-sphere, the Manolescu invariants $\alpha(Y,\s),$ $\beta(Y,\s), \gamma(Y,\s)$ and $\delta(Y,\s)$ are defined by $\alpha(\mathit{SWF}(Y,\s)),$ $\beta(\mathit{SWF}(Y,\s)),$ $\gamma(\mathit{SWF}(Y,\s))$ and $\delta(\mathit{SWF}(Y,\s))$, respectively.  We further define 
\[\underline{\delta}(Y,\s)=\underline{\delta}(\mathit{SWF}(Y,\s)) \; \text{and} \; \overline{\delta}(Y,\s)=\overline{\delta}(\mathit{SWF}(Y,\s)).\]
All these quantities are invariant under homology cobordism. 

\smallskip
\emph{Proof of Theorems \ref{thm:dvsz4d}-\ref{thm:z4dvsg}.}  These Theorems follow by applying Propositions \ref{prop:z2z4}-\ref{prop:dupdownfromg} to $\mathit{SWF}(Y,\s)$, and dualizing.

\end{defn}

\subsection{Equivariant Homology of subgroups of $G$}\label{subsec:main}
Here we make precise and prove Theorem \ref{thm:main}.  We first define generalized Fr{\o}yshov invariants.

Let $H\subseteq G$ be a Lie subgroup of $G$.  Note that $H^*(BG)$ is periodic; that is, cup product with $v\in H^*(BG)$ defines an isomorphism of $\f$-modules \[H^n(BG)\rightarrow H^{n+4}(BG)\] for all $n\geq 0$.  It turns out that $H^*(BH)$ is also periodic; fix $P \in H^*(BH)$ so that cup product with $P$ induces an isomorphism $H^*(BH)\rightarrow H^{*+\mathrm{deg}\, P}(BH)$.    

For $X$ a space of type SWF at level $s$, let $\iota: X^{S^1}\to X$ denote the inclusion map of the $S^1$-fixed point set, and let $\iota^*$ denote the induced map in Borel cohomology
\[
\iota^*:\tilde{H}^{*}_H(X) \to \tilde{H}^*_H(X^{S^1})=H^{*+s}(BH).
\]
\begin{defn}\label{def:genfroy}
For a homogeneous element $e$ of $H^*(BH)/P$ (with $\mathbb{Z}/\mathrm{deg}\, P$-grading) and $X\in \E$ we define the \emph{generalized Fr{\o}yshov invariant} $\delta_{H,e}(X)$ by:
\begin{equation}\label{eq:genfroy}
\frac{\min \{ m \equiv 2\mu(X)+\mathrm{deg}\, e \, \bmod{(\deg P)} \mid \exists x \in \tilde{H}^m_H(X),\; \iota^*x=P^ke, \text{ for some } k\}-\mathrm{deg}\; e}{2}.
\end{equation}  
\end{defn} 
The well-definedness of the quantity $\delta_{H,e}(X)$ is guaranteed by the Equivariant Localization Theorem.  It is apparent that all of $\alpha, \beta, \gamma, \delta$ and $\underline{\delta}$ and $\overline{\delta}$ are special cases of generalized Fr{\o}yshov invariants.

\begin{thm}\label{thm:mainunder}
Let $H \subset G$ a Lie subgroup, and let $\{\delta_{H,e}\}$ be the set of generalized Fr{\o}yshov invariants associated to $H$.  Then \[\{ \delta_{H,e} \}\subseteq \{\delta,\underline{\delta},\overline{\delta},\alpha,\beta,\gamma\},\]
where the generalized Fr{\o}yshov invariants are viewed as maps $\E \rightarrow \mathbb{Z}$.  Moreover, $\underline{\delta}(X)$ and $\overline{\delta}(X)$ are not generally determined by $\tilde{H}^G_*(X)$.\end{thm}
\begin{proof}
We refer to Example \ref{ex:weirdgkoszulexamples} for the last assertion, so we need only determine $\delta_{H,e}$.

First, consider strict subgroups $\mathbb{Z}/n= H \subset S^1$.  

If $n$ is odd, then $H^*(B\mathbb{Z}/n; \mathbb{Z}/2)\cong \f$, concentrated in degree $0$, and so there are no generalized Fr{\o}yshov invariants.  For $n$ even, $H^*(B\mathbb{Z}/n;\mathbb{Z}/2) \cong H^*(BS^1;\mathbb{Z}/2)$, and in particular the only associated generalized Fr{\o}yshov invariant is $\delta_{\mathbb{Z}/n,1}$.  The same argument as in Lemma \ref{lem:z2s1} shows $\delta_{\mathbb{Z}/n,1}=\delta$.  Thus the generalized Fr{\o}yshov invariants associated to a subgroup of $S^1$ are determined by $\delta$ and determine $\delta$.  
  
Next, consider a strict subgroup $H\subset G$ not contained in $S^1$ and not equal to $\mathbb{Z}/4$.  Then $H$ is a generalized quaternion group $Q_{4m}=\langle e^{\pi i /m}, j\rangle$ with $m\geq 2$.

First, say $m$ even.  We note $H^1(BQ_{4m})=\mathrm{Hom}\;(Q_{4m},\f)=\f^2$.  Then, since $H^1(BG)=\f$, we see that the Gysin sequence associated to the sphere bundle \[S^1 \rightarrow BQ_{4m}\rightarrow BG\] splits: 
\begin{equation}\label{eq:gysingsplitting}  
H^*(BQ_{4m})=H^*(BG) \oplus H^*(BG)[-1]
\end{equation}
as an $H^*(BG)$-module.  Recall $H^*(BG)$ acts on $H^*(BQ_{4m})$ by the map $p:H^*(BG)\rightarrow H^*(BQ_{4m})$.  

Let $r$ generate $H^1(BG)[-1]$ in the decomposition (\ref{eq:gysingsplitting}).  Then the homogeneous elements of $H^*(BQ_{4m})/(v)$ are
\[ 1,\;q,\;q+r,\;r, \;q^2,\;q^2+qr,\;qr,\; \mathrm{and} \; q^2r\]
Furthermore, we note from the definition of $\delta_{H,e}$ that, for $X\in \E$,
\begin{equation}\label{eq:froyorder}
\delta_{H,f}(X)\geq \delta_{H,e}(X) \; \mathrm{if} \; f\; \text{divides}\; e.
\end{equation}
Repeating the argument in the proof of Propositions \ref{prop:dfromg} and \ref{prop:dupdownfromg}, we see
\begin{equation}\label{eq:badgineqs}\delta_{r}(X)\geq \alpha(X),\; \delta_{q+r}(X)\geq \alpha(X),\;\delta_{qr}(X)\geq \beta(X),\;  \delta_{qr+q^2}(X)\geq \beta(X),\; \delta_{q^2r}(X)\geq \gamma(X),   \end{equation}
\begin{equation}\label{eq:badgineqs2}\delta_{1}(X)\leq \alpha(X),\; \delta_{q}(X)\leq \beta(X),\; \delta_{q^2}(X)\leq \gamma(X).\end{equation}
However, $\delta_{r}(X)\leq \delta_{1}(X)$ by (\ref{eq:froyorder}), so $\delta_{r}(X)\leq \delta_{1}(X)=\alpha(X)$.  Similarly, one obtains that all the inequalities in (\ref{eq:badgineqs}) and (\ref{eq:badgineqs2}) are equalities.  Thus $\delta_{H,e}(X)$ are in fact determined by $\alpha(X), \beta(X)$ and $\gamma(X)$.  This completes the proof for the $m$ even case.  

If $m$ is odd, we have $H^*(BQ_{4m})\cong H^*(B\mathbb{Z}/4)$.  The argument of Lemma \ref{lem:z2s1} then adapts to show that the generalized Fr{\o}yshov invariants of $Q_{4m}$ and $\mathbb{Z}/4$ agree. 
\end{proof}

Theorem \ref{thm:main} follows from Theorem \ref{thm:mainunder}, while Corollary \ref{cor:maincor} is a consequence of the proof of Theorem \ref{thm:mainunder}.  We close with an example showing that $\mathit{SWFH}^G_*(Y,\s)$, as an $H^*(BG)$-module, does not determine $\delta$, $\overline{\delta}$, or $\underline{\delta}$.  

\begin{examp}\label{ex:weirdgkoszulexamples}
There are pointed $G$-stable homotopy types $X_1$ and $X_2$ so that 
\[\tilde{H}^*_G(X_1)=\tilde{H}^*_G(X_2)=\bv_{8}\oplus \bv_1 \oplus \bv_2 \oplus \f^2_{3}\oplus \f_{4}\]
where $\bv_n$ denotes the $\f[v]$-module $\f[v]$, with grading shifted up by $n$, and $\f_n$ denotes a copy of $\f$ concentrated in degree $n$.  Furthermore,
\[\overline{\delta}(X_1)=\delta(X_1)=2, \; \underline{\delta}(X_i)=0,  \; \text{and} \; \overline{\delta}(X_2)=\delta(X_2)=3. \]
To specify the $q$-action, let $t_8, t_1,$ and $t_2$ be $\f[v]$-generators of $\bv_8$, $\bv_1$, and $\bv_2$ respectively, while $y_3,y'_3$ generate $\f^2_3$ and $y_4$ generates $\f_4$.  Then $qt_8=v^2t_1, qt_1=t_2$, $qt_2=y_3$ and $qy_3'=y_4$.  

We give a description of the chain complexes of $X_1$ and $X_2$ over $C^{CW}_*(G)=\f[s,j]/(sj=j^3s,s^2=j^4+1=0)$ where $\mathrm{deg}\;s=1,$ $\mathrm{deg}\; j=0$.  Indeed $C^{CW}_*(X_1)$ is $\f[f,x_1,x_3,x_4,x_5,y_3]$ with $\partial(x_1)=f$, $\partial(x_3)=(1+j)^3sx_1$, $\partial(x_4)=(1+j)x_3$, $\partial(x_5)=(1+j)x_4+sx_3$ and $\partial(y_3)=(1+j)^2sx_1$.  We have $C^{CW}_*(X_2)=\f[f,x_1,x_3,x_4,y_3,y_5]$ where the differentials are as before, and $\partial(y_5)=(1+j)^2sy_3$.  The calculation of the Manolescu invariants for both examples is an application of the techniques of \cite{betaseifert}, \cite{betasums}.  
\end{examp}

\bibliography{rmkpin.bib}
\bibliographystyle{plain}
\end{document}